\documentclass[11pt]{article}
\usepackage{amsmath,amssymb,amstext,dsfont,fancyvrb,float,fontenc,graphicx,subfigure, theorem}

\title{\Large\bf On multivalent solutions of \\ Riemann-Hilbert problem}
\author{\sc Vladimir Ryazanov}
\date{}

\cleardoublepage 
\usepackage[strict]{changepage}
\def\abstractname{Abstract -}   
\def\abstract{\begin{adjustwidth}{1cm}{1cm} \par    \footnotesize \noindent {\bf \abstractname}
\def\endabstract{ \end{adjustwidth} \smallskip }}

{\theorembodyfont{\itshape}\newtheorem{theorem}{Theorem}[section]}
{\theorembodyfont{\itshape}\newtheorem{proposition}{Proposition}[section]}
{\theorembodyfont{\itshape}}
{\theorembodyfont{\itshape}\newtheorem{lemma}{Lemma}[section]}
{\theorembodyfont{\itshape}}
{\theorembodyfont{\rm}}
{\theorembodyfont{\rm}}
{\theorembodyfont{\rm}}
{\theorembodyfont{\rm }}

\begin{document}
\maketitle
\vskip 1.5em

\vskip 1.5em
 \begin{abstract}
It is proved the existence of multivalent solutions for the
Riemann-Hilbert problem in the general settings of finitely
connected domains bounded by mutually disjoint Jordan curves,
mea\-su\-rab\-le coefficients and measurable boundary data. The
theorem is formulated in terms of harmonic measure and principal
asymptotic values. It is also given the corresponding reinforced
criterion for domains with rectifiable boundaries stated in terms of
the natural parameter and non\-tan\-gen\-tial limits. Furthemore, it
is shown that the dimension of the spaces of these solutions is
infinite.
 \end{abstract}

\begin{keywords} Riemann-Hilbert problem, Jordan curves, harmonic measures, principal
asymptotic values, nontangential limits.
\end{keywords}

\begin{MSC}
primary   31A05, 31A20, 31A25, 31B25, 35Q15; se\-con\-da\-ry 30E25,
31C05, 34M50, 35F45.
\end{MSC}

\section{Introduction}

This note is a continuation of the paper \cite{R1} where the
Riemann-Hilbert problem was resolved in these general settings for
simply connected domains and where you could find the brief history
of the question and the corresponding references, see also
\cite{R2}, \cite{R3} and \cite{R4}. Here, on the basis of \cite{R1}
and a theorem due to Poincare, see e.g. Section VI.1 in \cite{Go},
it is given a resolution of the problem for finitely connected
domains.

\bigskip


\section{The case of circular domains}

Let us start from the simplest kind of multiply connected domains.
Recall that a domain $D$ in $\overline{\Bbb C}=\Bbb C\cup \{
\infty\}$ is called {\bf circular} if its boundary consists of a
finite number of mutually disjoint circles and points. We call such
a domain {\bf nondegenerate} if its boundary consists only of
circles.

\medskip

\begin{theorem}{}\label{thRHD}{\it\, Let $\Bbb D_*$ be a bounded
nondegenerate circular domain and let $\lambda:\partial\mathbb
D_*\to\mathbb C$, $|\lambda (\zeta)|\equiv 1$, and
$\varphi:\partial\mathbb D_*\to\mathbb R$ be measurable functions.
Then there exist multivalent analytic functions $f:\mathbb
D_*\to\mathbb C$ such that
\begin{equation}\label{eqLIMD} \lim\limits_{z\to\zeta}\ \mathrm
{Re}\ \{\overline{\lambda(\zeta)}\cdot f(z)\}\ =\
\varphi(\zeta)\end{equation} along any nontangential path to a.e.
$\zeta\in\partial\mathbb D_*$.} \end{theorem}

\medskip

\begin{proof} Indeed, by the Poincare theorem, see e.g. Theorem VI.1 in
\cite{Go}, there is a locally conformal mapping $g$ of the unit disc
$\Bbb D=\{ z\in\Bbb C: |z|<1\}$ onto $\Bbb D_*$. Let $h:\Bbb
D_*\to\Bbb D$ be the corresponding multivalent analytic function
that is inverse to $g$. $\Bbb D_*$ without a finite number of cuts
is simply connected and hence $h$ has there only single-valued
branches that are extended to the boundary by the Caratheodory
theorem.

By Section VI.2 in \cite{Go}, $\partial\Bbb D$ except a countable
set of its points consists of a coun\-table collection of arcs every
of which is a one-to-one image of a circle in $\partial\Bbb D_*$
without its one point under an extended branch of $h$. Note that by
the reflection principle $g$ is conformally extended into a
neighborhood of every such arc and, thus, nontangential paths to its
points go into nontangential paths to the corresponding points of
circles in $\partial\Bbb D_*$ and inversely.

Setting $\Lambda = \lambda\circ g$ and $\Phi = \varphi\circ g$ with
the extended $g$ on the given arcs of $\partial\Bbb D$ we obtain the
suitable measurable functions on $\partial\Bbb D$. By Theorem 2.1 in
\cite{R1} or in \cite{R3} there exist analytic functions $F:\Bbb
D\to\Bbb C$ such that
\begin{equation}\label{F}
\lim\limits_{w\to\eta}\ \mathrm {Re}\
\{\overline{\Lambda(\eta)}\cdot F(w)\}\ =\ \Phi(\eta) \end{equation}
along any nontangential path to a.e. $\eta\in\partial\mathbb D$. By
the above arguments $f=F\circ h$ are desired multivalent analytic
solutions of (\ref{eqLIMD}).
\end{proof}

\bigskip

In particular, choosing $\lambda\equiv 1$ in (\ref{eqLIMD}), we
obtain the following statement.

\begin{proposition}\label{prDIR}
{\it\, Let $\Bbb D_*$ be a bounded nondegenerate circular domain and
let $\varphi:\partial \mathbb D_*\to\mathbb R$ be a measurable
function. Then there exist multivalent analytic functions $f:\mathbb
D_*\to\mathbb C$ such that
\begin{equation}\label{eqLIMDIRA} \lim\limits_{z\to\zeta}\ \mathrm
{Re}\
 f(z)\ =\ \varphi(\zeta)
\end{equation}
along any nontangential path to a.e. $\zeta\in\partial\mathbb
D_*$.}\end{proposition}

\bigskip

\section{Domains bounded by rectifiable Jordan curves}

To resolve the Riemann-Hilbert problem in the case of domains
bounded by a finite number of Jordan curves we should extend to the
case the known results of Caratheodory (1912), Lindel\"of (1917), F.
and M. Riesz (1916) and Lavrentiev (1936).

\begin{lemma}{}\label{multi}{\it\, Let $D$ be a bounded domain in
$\mathbb C$ whose boundary consists of a finite number of mutually
disjoint Jordan curves, $\mathbb D_*$ be a bounded nondegenerate
circular domain in $\mathbb C$ and let $\omega : D\to\mathbb D_*$ be
a conformal mapping. Then

\medskip

(i) $\omega$ can be extended to a homeomorphism of $\overline D$
onto $\overline{\mathbb D_*}$;

\medskip

(ii) $\arg\ [\omega(\zeta)-\omega(z)]-\arg\ [\zeta-z]\to\mathrm
const$ as $z\to\zeta$ whenever $\partial D$ has a tangent at
$\zeta\in\partial D$;

\medskip

(iii) for rectifiable $\partial D$, $\mathrm{length}\
\omega^{-1}(E)=0$ whenever $|E|=0$, $E\subset\partial\mathbb D_*$;

\medskip

(iv) for rectifiable $\partial D$, $|\omega({\cal E})|=0$ whenever
$\mathrm{length}\ {\cal E}=0$, ${\cal E}\subset\partial D$.

}\end{lemma}

\medskip

\begin{proof} (i) Indeed, we are able to transform $\mathbb
D_*$ into a simply connected domain $\mathbb D^*$ through a finite
sequence of cuts. Thus, we come to the desired conclusion applying
the Caratheodory theorems to simply connected domains $\mathbb D^*$
and $D^*:=\omega^{-1}(\mathbb D^*)$, see e.g. Theorem 9.4 in
\cite{CL} and Theorem II.C.1 in \cite{Ko}.

(ii) In the construction from the previous item, we may assume that
the point $\zeta$ is not the end of the cuts in $D$ generated by the
cuts in $\mathbb D_*$ under the extended mapping $\omega^{-1}$.
Thus, we obtain to the desired conclusion twice applying the
Caratheodory theorems, the reflection principle for conformal
mappings and the Lindel\"of theorem for the Jordan domains, see e.g.
Theorem II.C.2 in \cite{Ko}.

Points (iii) and (iv) are proved similarly to the last item on the
basis of the corresponding results of F. and M. Riesz and Lavrentiev
for Jordan domains with rectifiable boundaries, see e.g. Theorem
II.D.2 in \cite{Ko}, and \cite{L}, see also the point III.1.5 in
\cite{P}.
\end{proof}

\begin{theorem}{}\label{thRHR}{\it\, Let $D$ be a bounded domain in
$\mathbb C$ whose boundary consists of a finite number of mutually
disjoint rectifiable Jordan curves and $\lambda:\partial D\to\mathbb
C$, $|\lambda (\zeta)|\equiv 1$, and $\varphi:\partial D\to\mathbb
R$ be measurable functions with respect to the natural parameter on
$\partial D$. Then there exist multivalent analytic functions
$f:\mathbb D\to\mathbb C$ such that along any nontangential path
\begin{equation}\label{eqLIM} \lim\limits_{z\to\zeta}\ \mathrm {Re}\
\{\overline{\lambda(\zeta)}\cdot f(z)\}\ =\ \varphi(\zeta)
\quad\quad\quad\mbox{for}\ \ \mbox{a.e.}\ \ \ \zeta\in\partial D
\end{equation}
with respect to the natural parameters of the boundary components of
$D$.}
\end{theorem}

\medskip

\begin{proof} This case is reduced to the case of a bounded nondegenerate
circular domain $\mathbb D_*$ in the following
way. First, there is a conformal mapping $\omega$ of $D$ onto a
circular domain $\mathbb D_*$, see e.g. Theorem V.6.2 in \cite{Go}.
Note that $\mathbb D_*$ is not degenerate because isolated
singularities of conformal mappings are removable that is due to the
well-known Weierstrass theorem, see e.g. Theorem 1.2 in \cite{CL}.
Applying in the case of need the inversion with respect to a
boundary circle of $\mathbb D_*$, we may assume that $\mathbb D_*$
is bounded.

\medskip

By point (i) in Lemma \ref{multi} $\omega$ can be extended to a
homeomorphisms of $\overline D$ onto $\overline{\mathbb D_*}$. If
$\partial D$ is rectifiable, then by point (iii) in Lemma
\ref{multi} $\mathrm{length}\ \omega^{-1}(E)=0$ whenever
$E\subset\partial\mathbb D_*$ with $|E|=0$, and by (iv) in Lemma
\ref{multi}, conversely, $|\omega({\cal E})|=0$ whenever ${\cal
E}\subset\partial D$ with $\mathrm{length}\ {\cal E}=0$.

\medskip

In the last case $\omega$ and $\omega^{-1}$ transform measurable
sets into measurable sets. Indeed, every measurable set is the union
of a sigma-compact set and a set of measure zero, see e.g. Theorem
III(6.6) in \cite{S}, and continuous mappings transform compact sets
into compact sets. Thus, a function $\varphi:\partial D\to\mathbb R$
is measurable with respect to the natural parameter on $\partial D$
if and only if the function
$\Phi=\varphi\circ\omega^{-1}:\partial\mathbb D_*\to\mathbb R$ is
measurable with respect to the linear measure on $\partial\mathbb
D_*$.

\medskip

By point (ii) in Lemma \ref{multi}, if $\partial D$ has a tangent at
a point $\zeta\in\partial D$, then $\arg\
[\omega(\zeta)-\omega(z)]-\arg\ [\zeta-z]\to\mathrm const$ as
$z\to\zeta$. In other words, the conformal images of sectors in $D$
with a vertex at $\zeta$ is asymptotically the same as sectors in
$\mathbb D_*$ with a vertex at $w=\omega(\zeta)$. Thus,
nontangential paths in $D$ are transformed under $\omega$ into
nontangential paths in $\mathbb D_*$ and inversely. Finally, a
rectifiable Jordan curve has a tangent a.e. with respect to the
natural parameter and, thus, Theorem \ref{thRHR} follows from
Theorem \ref{thRHD}.
\end{proof}

\bigskip

In particular, choosing $\lambda\equiv 1$ in (\ref{eqLIM}), we obtain the following statement.

\begin{proposition}\label{prDR}
{\it\, Let $D$ be a bounded domain in $\mathbb C$ whose boundary
consists of a finite number of mutually disjoint rectifiable Jordan
curves and let $\varphi:\partial D\to\mathbb R$ be measurable. Then
there exist multivalent analytic functions $f:D\to\mathbb C$ such
that
\begin{equation}\label{eqLIMDRA} \lim\limits_{z\to\zeta}\ \mathrm
{Re}\
 f(z)\ =\ \varphi(\zeta)
\quad\quad\quad\mbox{for}\ \ \mbox{a.e.}\ \ \ \zeta\in\partial D
\end{equation}
along any nontangential path with respect to the natural parameters
of the boundary components of $\partial D$.}\end{proposition}

\bigskip

\section{Domains bounded by arbitrary Jordan curves}

See detailed comments on terms of {\bf harmonic measures and
principal asymptotic values} in the paper \cite{R1} or in \cite{R3}.


\begin{theorem}{}\label{thRH1}{\it\, Let $D$ be a bounded domain in $\mathbb C$ whose
boundary consists of a finite number of mutually disjoint Jordan
curves and let $\lambda:\partial D\to\mathbb C$, $|\lambda
(\zeta)|\equiv 1$, and $\varphi:\partial D\to\mathbb R$ be
measurable functions with respect to harmonic measures in $D$. Then
there exist multivalent analytic functions $f:\mathbb D\to\mathbb C$
such that
\begin{equation}\label{eqLIMA} \lim\limits_{z\to\zeta}\ \mathrm {Re}\
\{\overline{\lambda(\zeta)}\cdot f(z)\}\ =\ \varphi(\zeta)
\quad\quad\quad\mbox{for}\ \ \mbox{a.e.}\ \ \ \zeta\in\partial D
\end{equation}
with respect to harmonic measures in $D$ in the sense of the unique
principal asymptotic value.}
\end{theorem}

\begin{proof} By the reasons of the first item in the proof of Theorem \ref{thRHR},
there is a conformal mapping $\omega$ of $D$ onto a bounded
nondegenerate circular domain $\mathbb D_*$ in $\mathbb C$. Set
$\Lambda = \lambda\circ\Omega$ and $\Phi = \varphi\circ\Omega$ where
$\Omega :=\omega^{-1}$ extended to $\partial\Bbb D_*$ by point (i)
in Lemma \ref{multi}.

Note that harmonic measure zero is invariant under conformal
mappings. Thus, arguing as in the third item of the proof to Theorem
\ref{thRHR}, we conclude that the functions $\Lambda$ and $\Phi$ are
measurable with respect to harmonic measures in $\mathbb D_*$.

By Theorem \ref{thRHD} there exist multivalent analytic functions
$F:\Bbb D_*\to\Bbb C$ such that
$$
\lim\limits_{w\to\eta}\ \mathrm {Re}\
\{\overline{\Lambda(\eta)}\cdot F(w)\}\ =\ \Phi(\eta)
$$
along any nontangential path to a.e. $\eta\in\partial\mathbb D_*$.

By the construction the functions $f:=F\circ\omega$ are desired
multivalent analytic solutions of (\ref{eqLIMA}) in view of the
Bagemihl theorem, see e.g. Theorem 2 in \cite{B} or Theorem III.1.8
in \cite{No}. \end{proof}

\bigskip

In particular, choosing $\lambda\equiv 1$ in (\ref{eqLIMA}), we
obtain the following consequence.

\medskip
\begin{proposition}\label{prDA}
{\it\, Let $D$ be a bounded domain in $\mathbb C$ whose boundary
consists of a finite number of mutually disjoint Jordan curves and
let $\varphi:\partial D\to\mathbb R$ be a measurable function with
respect to harmonic measures in $D$. Then there exist multivalent
analytic functions $f:D\to\mathbb C$ such that
\begin{equation}\label{eqLIMDA} \lim\limits_{z\to\zeta}\ \mathrm
{Re}\
 f(z)\ =\ \varphi(\zeta)
\quad\quad\quad\mbox{for}\ \ \mbox{a.e.}\ \ \ \zeta\in\partial D
\end{equation}
with respect to harmonic measures in $D$ in the sense of the unique
principal asymptotic value.}\end{proposition}

\bigskip

\section{On the dimension of spaces of solutions}

By the Lindel\"of maximum principle, see e.g. Lemma 1.1 in
\cite{GM}, it follows the uniqueness theorem for the Dirichlet
problem in the class of bounded harmonic functions on the unit disk.
Our multivalent analytic solutions are generally speaking not
bounded and we have the new phenomena.

\begin{theorem}{}\label{thDIM1}{\it\, The spaces of multivalent analytic solutions
in Theorems \ref{thRHD}, \ref{thRHR} and \ref{thRH1}
and in Propositions \ref{prDIR}, \ref{prDR} and \ref{prDA} have the
infinite dimension.}
\end{theorem}

\medskip

\begin{proof} By Theorem 5.1 in \cite{R1} or in \cite{R3} the space
of solutions of the problem (\ref{F}) has the infinite infinite
dimension. Thus, the conclusion follows by the construction of
solutions in the given theorems.
\end{proof}

\bigskip



 { \footnotesize
\medskip
\medskip
 \vspace*{1mm}

\noindent {\it Vladimir Ryazanov}\\
Institute of Applied Mathematics and Mechanics \\
National Academy of Sciences of Ukraine \\
1 Dobrovol'skogo Str., 84100 Slavyansk, Ukraine\\
E-mail: {\tt vl.ryazanov1@gmail.com}}

 \end{document}